\documentclass[11pt,reqno]{amsart}

\usepackage[T1]{fontenc}
\usepackage{mathptmx,microtype,hyperref,graphicx,dsfont}
\usepackage{amsthm,amsmath,amssymb}
\usepackage[foot]{amsaddr}
\usepackage{enumitem}
\usepackage[nameinlink,capitalise]{cleveref}
\usepackage[font=footnotesize,margin=2cm]{caption}
\usepackage{cite}

\crefname{theorem}{Theorem}{Theorems}
\crefname{thm}{Theorem}{Theorems}
\crefname{mainthm}{Theorem}{Theorems}
\crefname{lemma}{Lemma}{Lemmas}
\crefname{lem}{Lemma}{Lemmas}
\crefname{remark}{Remark}{Remarks}
\crefname{claim}{Claim}{Claims}
\crefname{subclaim}{Sub-claim}{Sub-claims}
\crefname{prop}{Proposition}{Propositions}
\crefname{proposition}{Proposition}{Propositions}
\crefname{defn}{Definition}{Definitions}
\crefname{corollary}{Corollary}{Corollaries}
\crefname{conjecture}{Conjecture}{Conjectures}
\crefname{question}{Question}{Questions}
\crefname{chapter}{Chapter}{Chapters}
\crefname{section}{Section}{Sections}
\crefname{figure}{Figure}{Figures}
\newcommand{\crefrangeconjunction}{--}

\theoremstyle{plain}

\newtheorem{thm}{Theorem}
\newtheorem*{thm*}{Theorem}
\newtheorem{lemma}[thm]{Lemma}

\newtheorem{corollary}[thm]{Corollary}
\newtheorem{prop}[thm]{Proposition}

\theoremstyle{definition}

\theoremstyle{remark}

\renewcommand{\P}{{\bf P}}
\newcommand{\E}{{\bf E}}

\newcommand{\Z}{{\mathbb Z}}


\author[B. Kolesnik]{Brett Kolesnik}

\address{Department of Statistics, University of Oxford}

\email{brett.kolesnik@stats.ox.ac.uk}

\begin{document}

\title[The asymptotic number of score sequences]
{
The asymptotic number of score sequences
}

\begin{abstract}
A tournament on a graph is an orientation of its edges. 
The score sequence lists the in-degrees in non-decreasing order. 
Works by Winston and Kleitman (1983) and 
Kim and Pittel (2000) showed that 
the number $S_n$ of score sequences on the complete
graph $K_n$ satisfies 
$S_n=\Theta(4^n/n^{5/2})$. 
By combining 
a recent recurrence relation for $S_n$ in terms 
of the Erd\H{o}s--Ginzburg--Ziv numbers $N_n$
 with the 
limit theory for discrete 
infinitely divisible
distributions, we observe that 
$n^{5/2}S_n/4^n\to e^\lambda/2\sqrt{\pi}$, where 
$\lambda=\sum_{k=1}^\infty  N_k/k4^k$. 
This limit agrees numerically with the 
asymptotics of $S_n$ conjectured by Tak\'acs (1986). 
We also 
identify the asymptotic number 
of strong score sequences,   
and show that the number of 
irreducible subscores in a random 
score sequence converges in distribution to a 
shifted 
negative
binomial with parameters $r=2$ and $p=e^{-\lambda}$. 
\end{abstract}

\maketitle

\vspace{-0.15cm}

\section{Introduction}\label{S_intro}

A {\it tournament} on a graph $G$
is an orientation of its edges. Informally, 
each vertex is a team and each 
pair of vertices
joined by an edge plays a game. Afterwards, the edge 
is directed towards
the winner. The {\it score sequence} lists the total number 
of wins by each team in non-decreasing order.
For a survey, see, e.g., Harary and Moser
\cite{HM66} and Moon \cite{Moo68}.

Let $S_n$ denote the number of score sequences
on the complete graph $K_n$. 
Calculations related to $S_n$ appear
in the literature as early as MacMahon~\cite{Mac20}. 
The investigation of its asymptotics began
with the unpublished work of Erd{\H{o}}s and Moser (see \cite{Moo68}) which 
showed  
$c/n^{9/2}\le S_n/4^n\le C/n^{3/2}$. 
Winston and Kleitman \cite{WK83} 
(cf.\ Kleitman \cite{Kle70})
proved 
$S_n/4^n\ge c/n^{5/2}$, and argued that  
an upper bound of the same order follows, supposing that the largest 
$q$-Catalan number 
is $O(4^n/n^3)$. 
This fact was finally verified by 
Kim and Pittel \cite{KP00}, and hence $S_n=\Theta(4^n/n^{5/2})$, 
as conjectured by Moser \cite{Mos71}. 

In this note, we observe the following. 
Let $N_n$ denote the number of subsets of $\{1,2,\ldots 2n-1\}$
that sum to a multiple of $n$. 

\begin{thm}\label{T_main}
As $n\to\infty$, we have that 
\begin{equation}\label{E_asySn}
S_n\sim 
\frac{N_n}{n}
\exp\left(\sum_{k=1}^\infty\frac{N_k}{k 4^k}\right).
\end{equation}
\end{thm}

Here, as usual, $f(n)\sim g(n)$ means that $f(n)/g(n)\to1$, 
as $n\to\infty$. 
The quantity in the exponential will appear often
throughout, so we 
put
\[
\lambda=\sum_{k=1}^\infty\frac{N_k}{k 4^k}.
\]

We note that, from a geometric point of view, 
$S_n$ is the number 
of non-decreasing lattice points in the permutahedron
$\Pi_{n-1}$, the convex hull of all permutations
of the vector $v_n=(0,1,\ldots,n-1)$. 
Note that $v_n$ is the score sequence for the tournament
on $K_n$ 
in which each team $i$ wins against all teams $j<i$
of smaller index 
(and loses against all other teams $j>i$). 
This polytope is equivalently obtained as the graphical zonotope of $K_n$, 
more specifically, the (translated) Minkowski sum 
of line segments 
$v_n+\sum_{i<j}[0,e_i-e_j]$, where
$e_i$ are standard basis vectors. 
See, e.g., 
Rado \cite{R52}, Landau \cite{Lan53}, Stanley \cite{S80}, 
and Ziegler \cite{Z95} for more details on these connections. 
See also, e.g., Postnikov \cite{Pos09} (and references therein) for results
on counting lattice points in the permutahedron
and other polytopes. 
Let us also note that $S_n$ is the number of lattice points in the partitioned permutahedron 
$P_{A_{n-1}}([n-1])$ in the recent work of 
Horiguchi, Masuda,  Shareshian and Song \cite{HMSS21}.

\begin{figure}[h]
\centering
\includegraphics[scale=0.3]{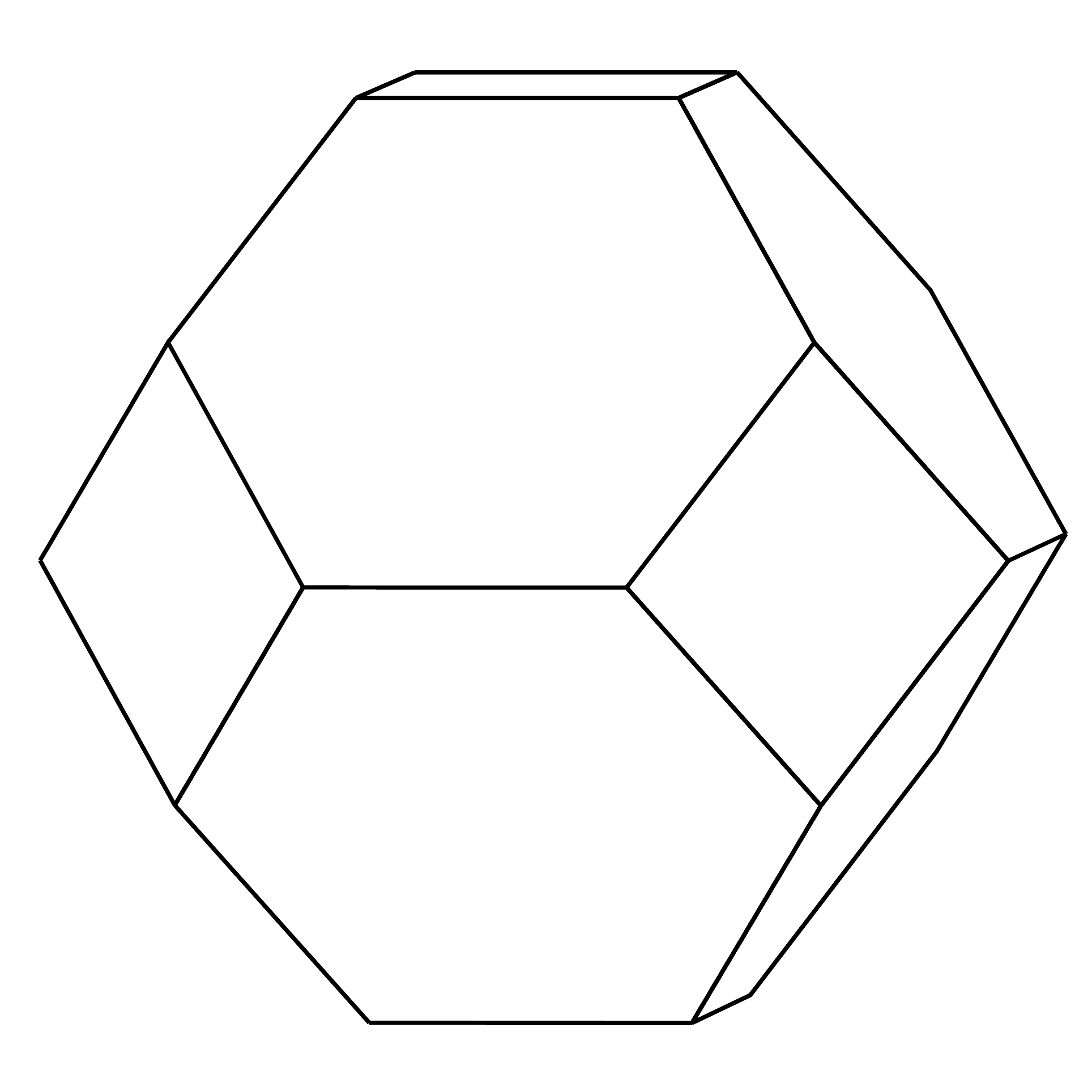}
\caption{
The permutahedron $\Pi_3$. 
}
\label{F_perm}
\end{figure}

An intriguing 
connection between $S_n$ and $N_n$ 
was conjectured by Hanna \cite{A000571} 
and recently proved by Claesson, Dukes, Frankl\'in and Stef\'ansson 
\cite[Corollary 12]{CDFS22}. 
As it turns out,   
\begin{equation}\label{E_CDFS22}
nS_n=\sum_{k=1}^n  N_k S_{n-k}. 
\end{equation}
This follows as a corollary of the main result in \cite[Theorem 4]{CDFS22} 
which establishes a relationship between generating functions
related to the numbers $S_n$, $N_n$ and 
``pointed''
score sequences (with a distinguished index). 
Another consequence of their main result is a beautiful connection
between score sequences and the Catalan numbers, see \cite[Corollary 11]{CDFS22}. 

The potential for a relationship between $S_n$ and $N_n$ is perhaps
not so obvious at first blush, however, let us 
remark here that, as observed in \cite[p.\ 212]{WK83}, 
each score sequence $s\in \Z^n$ is associated
with a subset $\{s_1+1,s_2+2,\ldots,s_n+n\}$
of $\{1,2,\ldots,2n-1\}$, which sums to $n^2$. 
Informally, if the sequence $s$ is drawn as a ``bar graph''
consisting of upward and rightward unit steps, then this 
subset contains the ``times'' at which 
the walk takes a rightward step, thus ``completing a bar.''

The relationship \eqref{E_CDFS22} 
implies that the sequence $\{S_0,S_1,\ldots\}$ is 
{\it infinitely divisible} (see \cref{S_ID} below). 
The theory of infinitely divisible distributions  
began with de Finetti \cite{deF29}. 
See, e.g., Wright~\cite{Wri67a}, Hawkes and Jenkins~\cite{HJ78}, 
van Harn~\cite{vHa78}, 
Embrechts and Hawkes~\cite{EH82}
and Steutel and van Harn~\cite{SvH04} 
for background and results in the discrete case
which are pertinent to the current article. 
Often,
when two sequences are related in this way, 
asymptotic information can be transferred between
them. 
The purpose of this note is to 
point out that the asymptotics of $S_n$ can be obtained
from those of $N_n$, via the limit theory developed in 
\cite{HJ78,EH82}. See \cref{S_proofs} for the proof
of \cref{T_main}.

\subsection{Tak\'acs' conjecture}
A well-known result of 
Erd{\H{o}}s, Ginzburg and Ziv \cite{EGZ61} states that
{\it every} set $I$ of $2n-1$ integers has a subset of size $n$
which sums to a multiple of $n$. 
Alekseyev \cite{Ale08} (cf.\ Chern \cite{Che19}) showed, 
in the special case $I=\{1,2,\ldots 2n-1\}$ that the integers
are consecutive, that
the number $N_n$ of such subsets is equal to 
\begin{equation}\label{E_Ale08}
N_n
=\sum_{k=1}^{n}\frac{(-1)^{n+d}}{2n}{2d\choose d},
\end{equation}
where $d={\rm gcd}(n,k)$. 
In \cref{S_constant}, we 
use this together with \cref{T_main} to 
obtain the following 
asymptotics, in line with those conjectured by 
Tak\'acs \cite[p.~136]{Tak86} (cf.\ Kotesovec \cite{A000571}).

\begin{corollary} 
\label{C_main}
As $n\to\infty$, we have that 
\[
\frac{n^{5/2}S_n}{4^n}
\to 
\frac{e^\lambda}{2\sqrt{\pi}}=
0.392\ldots.
\]
\end{corollary}

\subsection{Irreducible subscores}\label{S_irr}

In a natural way, any tournament (or score sequence) 
can be decomposed into a series of irreducible 
parts (see, e.g., Moon and Moser \cite{MM62} and Wright \cite{W70}). 
Landau \cite{Lan53}  showed that 
a sequence
$s\in\Z^n$ is a tournament score sequence
on $K_n$ if and only if 
$\sum_{i=1}^k s_n\ge {k\choose2}$
with equality when $k=n$. In other words, if and only if 
$s$ is majorized (see, e.g., 
Marshall, Olkin and Arnold~\cite{MOA11}) by $v_n$. 
We let $S_{n,m}$ denote the number of 
score sequences 
with exactly $m$ {\it irreducible subscores},
that is, points $k$ at which $\sum_{i=1}^k s_i={k\choose 2}$
takes the smallest possible value. Hence $S_n=\sum_{m=1}^n S_{n,m}$.

In \cref{S_ID}, we show that 
 $\{S_0,S_1,\ldots\}$ is a {\it renewal sequence}.
In \cref{S_exp}, we use this to deduce that, asymptotically,  
the exponential term in \eqref{E_asySn} has the following 
probabilistic interpretation. 

\begin{corollary}\label{T_irr}
Let ${\mathcal I}_n$ denote the number of irreducible 
subscores 
in a uniformly random 
tournament score sequence on $K_n$.  
Then, as $n\to\infty$, the expected inverse number
of such subscores
satisfies 
\[
\E\left(\frac{1}{{\mathcal I}_n}\right)
=\sum_{m=1}^n\frac{1}{m}\frac{S_{n,m}}{S_n}
=\frac{N_n}{nS_n}\to
e^{-\lambda}
=0.718\ldots.
\]
\end{corollary}

\subsection{Strong score sequences}\label{S_irr}

Let $S_{n,1}$ denote the number of {\it strong (irreducible) score sequences}. 
The reason for the name is that a score sequence $s$ of a strongly connected
tournament (in which each pair of vertices are in a directed cycle) is irreducible, 
that is, $\sum_{i=1}^k s_i>{k\choose 2}$
for all $k<n$. 
In other words, $s$ is strictly majorized by $v_n$. 
See, e.g., \cite[Theorem 9]{HM66}. 

Combining the ``reverse renewal theorem'' 
in Alexander and Berger \cite{AB16}
with the results above, 
we obtain the following. 

\begin{corollary}\label{T_irr2}
As $n\to\infty$, we have that 
\[
\frac{S_{n,1}}{S_n}
\to 
e^{-2\lambda}
= 0.515\ldots, 
\]
and so 
\[
\frac{n^{5/2} S_{n,1}}{4^n}
\to 
\frac{e^{-\lambda}}{2\sqrt{\pi}}
=0.202\ldots.
\]
\end{corollary}

These asymptotics agree with those conjectured
by Kotesovec \cite{A351822}, as discussed recently in 
Stockmeyer \cite[Section 6]{S22}.
Note that, for all sufficiently large $n$, more than half of all score sequences are strong. 

\subsection{Limiting distribution}

With these asymptotics at hand, 
we can then apply the powerful local limit theory of 
Chover, Ney and Wainger
\cite{CNW73} to, more generally, obtain the asymptotic
distribution of the number of irreducible subscores
in a random score sequence. 

\begin{corollary}\label{T_Snm}
Let ${\mathcal I}_n$ denote the number of irreducible 
subscores 
in a uniformly random 
tournament score sequence on $K_n$. 
Then
\[
{\mathcal I}_n
\xrightarrow{d}
1+{\mathcal N}_\lambda, 
\]
where 
${\mathcal N}_\lambda$ is a negative binomial (Pascal) random variable
with parameters $r=2$ and $p=e^{-\lambda}$. 
That is, for any $m\ge1$, 
\[
\P({\mathcal I}_n=m)=
\frac{S_{n,m}}{S_n}
\to 
m(1-e^{-\lambda})^{m-1}e^{-2\lambda}
=\P({\mathcal N}_\lambda=m-1),
\]
as $n\to\infty$. In particular, 
\[
\frac{n^{5/2} S_{n,m}}{4^n}
\to 
\frac{m(1-e^{-\lambda})^{m-1}e^{-\lambda}}{2\sqrt{\pi}}.
\]
\end{corollary}

We note that the mean 
\[
\E({\mathcal I}_n)\to2(1-e^{-\lambda})e^{\lambda}+1
=1.782\ldots
\]
and variance 
\[
{\bf V}({\mathcal I}_n)\to2(1-e^{-\lambda})e^{2\lambda}
=1.088\ldots.
\]
See \cref{F_ld}. 
Let us also note that, equivalently, ${\mathcal N}_\lambda$ can be expressed
as a compound random variable 
$\sum_{i=1}^N X_i$, where 
$N$ is Poisson with rate $-2\log(1-e^{-\lambda})$ and the $X_i$ 
are independent and identically logarithmically distributed as  
\[
\P(X_i=k)=-\frac{e^{-\lambda k}}{k\log(1-e^{-\lambda})},\quad
k\ge1.
\]

\begin{figure}[h]
\centering
\includegraphics[scale=0.35]{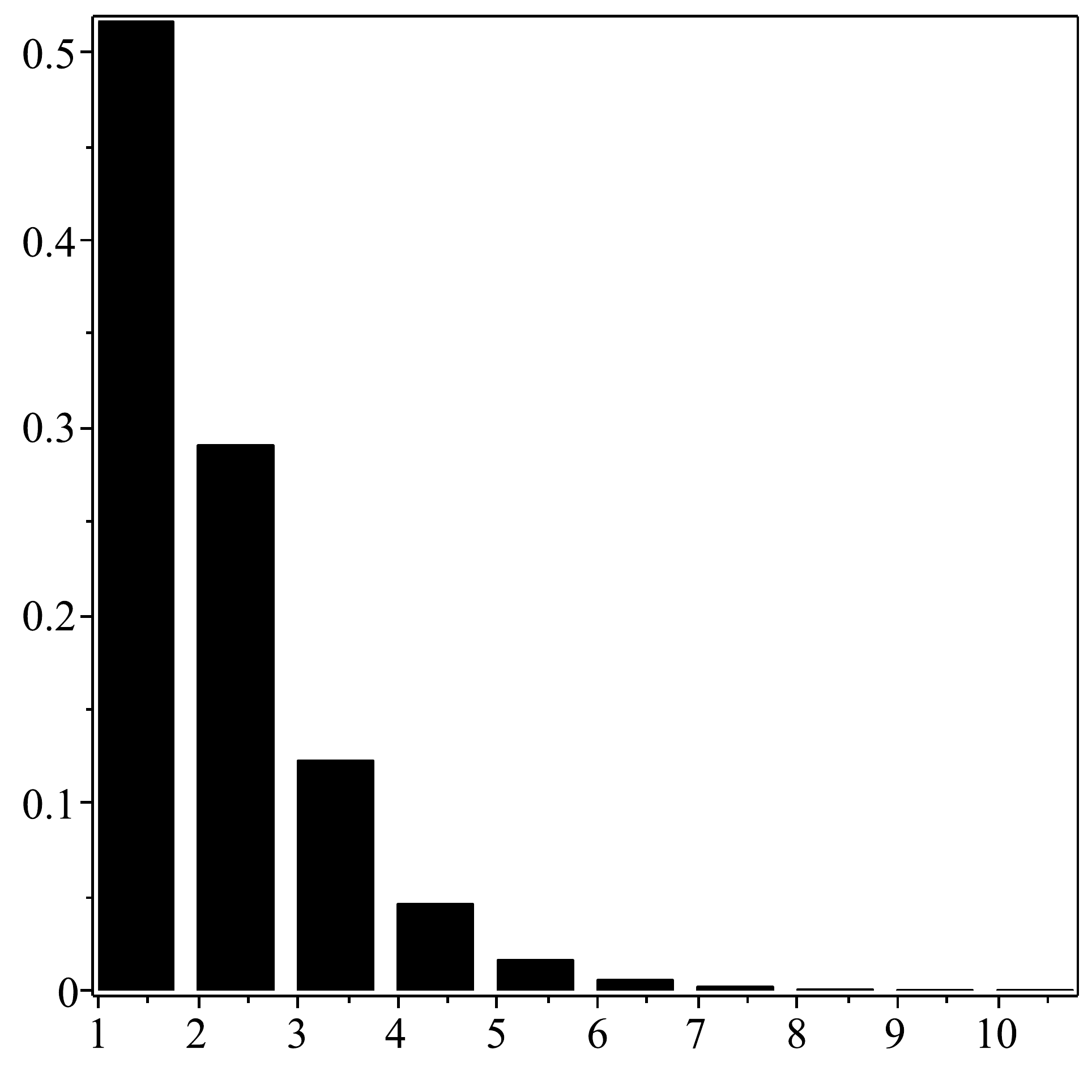}
\caption{
The limiting subscore distribution, negative
binomial with parameters $r=2$ and $p=e^{-\lambda}$.
}
\label{F_ld}
\end{figure}

\subsection{Acknowledgments}

We thank Jim Pitman and Mario Sanchez 
for helpful discussions.

\section{Infinite divisibility}
\label{S_ID}

Following \cite{HJ78}, we call 
a positive sequence $\{1=a_0,a_1,\ldots\}$ {\it infinitely divisible} 
if for any integer $r\ge1$ there is a non-negative 
sequence $\{b_0,b_1,\dots\}$ such that 
$a_n=b_n^{*r}$, where $b_n^{*r}$ is the $r$th {\it convolution
power}, defined inductively by $b_n^{*1}=b_n$ and 
$b_n^{*r}=\sum_{k=0}^n b_{k}b_{n-k}^{*(r-1)}$, for $r>1$. 
In \cite[Theorem~2.1]{HJ78} (cf., e.g., Katti~\cite{Kat67}, 
Steutel~\cite{Ste71} and van Harn~\cite{vHa78}) 
it is shown that 
a positive sequence $\{1=a_0,a_1,\ldots\}$  is 
infinitely divisible if and only if 
\begin{equation}\label{E_HJ78}
na_n=\sum_{k=1}^n \hat a_k a_{n-k}
\end{equation}
for some other non-negative 
sequence $\{0=\hat a_0,\hat a_1,\ldots\}$.
We call $\hat a_n$  
the {\it log transform} 
of $a_n$, since their generating functions 
$A(x)=\sum_{n=0}^\infty a_n x^n$ and $\hat A(x)=\sum_{n=0}^\infty \hat a_n x^n$
satisfy 
\[
\hat A(x)=x\frac{d}{dx}\log A(z), 
\]
and so,  
\[
A(x)=
\exp\left(\sum_{k=1}^\infty \hat a_kx^k/k\right).
\]

As discussed above, 
it was recently proved in \cite{CDFS22}
that log the transform of $S_n$ is $\hat S_n=N_n$.
We note that, in \cite{CDFS22}, $\hat A(x)$ is called the 
``logarithmic pointing'' of $A(x)$
and it is observed  that, 
when \eqref{E_HJ78} holds, 
a certain closed-form expression 
for $a_n$ can be obtained in terms of the $\hat a_k$. 
In \cite[Corollary~14]{CDFS22}, it is shown that 
\[
S_n=\frac{1}{n!}\sum_{\pi\in {\rm Sym}(n)}\prod_{\ell\in C(\pi)}
N_\ell,
\]
where ${\rm Sym}(n)$ is the symmetric group on 
$\{1,2,\ldots,n\}$ and $C(\pi)$ is the sequence of 
cycle lengths in $\pi$. 
However, while this 
gives a closed-form expression for $S_n$, 
it is not clear to us
if the asymptotics of $S_n$ can be easily read off 
from this formula.

\section{Proofs}\label{S_proofs}

Our main result follows quite simply by the limit theorems in \cite{HJ78,EH82}. 

\begin{proof}[Proof of \cref{T_main}]
Put 
\[
\alpha_n=\frac{S_n}{4^n},\quad\quad \beta_n=\frac{N_n}{n4^n}.
\]
Then, by \eqref{E_CDFS22}, we obtain
\begin{equation}\label{E_EH82}
n\alpha_n=\sum_{k=1}^n k\beta_k\alpha_{n-k}.
\end{equation}

In \cite{EH82}  it is shown that, 
for sequences of this form, the asymptotics of 
$\beta_n$ can be transferred to $\alpha_n$,  
under the following conditions.

\begin{thm}[{\hspace{1sp}\cite[Theorem 1]{EH82}}]
\label{T_EH82}
Suppose that $\{1=\alpha_0,\alpha_1,\ldots\}$ and
$\{\beta_1,\beta_2,\ldots\}$ are positive sequences 
satisfying 
\eqref{E_EH82}
and that $\lambda=\sum_{k=1}^\infty \beta_k<\infty$. 
Put $\beta_0=0$. 
Then, if $\beta_{n+1}\sim \beta_n$ and 
$\beta_n^{*2}\sim 2\lambda \beta_n$, it follows that 
$\alpha_n\sim e^\lambda\beta_n$. 
\end{thm}

In order to apply this result, we first note that, in this instance, 
we have by \eqref{E_Ale08} (see \cref{L_tech}
in \cref{A_tech} below) 
that  
\[
\beta_n=\frac{1}{2\sqrt{\pi}}\frac{1}{n^{5/2}}\left(1+\Theta\left(\frac{1}{n}\right)\right).
\]
The condition $\beta_{n+1}\sim \beta_n$
is clear. 
Next, we note that 
\[
\frac{\beta_n^{*2}}{2\beta_n}\ge \sum_{k=1}^{\lfloor n/2\rfloor}\beta_k \frac{\beta_{n-k}}{\beta_n}
\ge(1+o(1))\sum_{k=1}^{\lfloor n/\log n\rfloor}\beta_k\to \lambda
\]
and 
\[
\frac{\beta_n^{*2}}{2\beta_n}\le 
\sum_{k=1}^{\lceil n/2\rceil}\beta_k \frac{\beta_{n-k}}{\beta_n}
\le (1+o(1))\sum_{k=1}^{\lfloor n/\log n\rfloor}\beta_k
+O\left(\frac{(\log n)^{5/2}}{n^{3/2}}\right)
\to\lambda.
\]
Hence the theorem applies, and 
\cref{T_main} follows.
\end{proof}

Note that, in our application of \cref{T_EH82}, 
the infinitely divisible distribution at play is 
\[
p_n=e^{-\lambda}\alpha_n=\frac{S_n}{4^n}\exp\left(-\sum_{k=1}^\infty\frac{N_k}{k4^k}\right), 
\quad n\ge0, 
\]
which we call the {\it tournament distribution}.
See \cref{F_td}. 
Recall that 
an infinitely divisible distribution on 
the non-negative integers is compound 
Poisson, that is, distributed as a random
sum 
$\sum_{i=1}^M X_i$ 
of $M$ independent and identically
distributed random variables $X_i$, 
which are moreover independent of $M$. 
See, e.g., Feller~\cite[Section XII.2, p.\ 290]{Fel68}. 
In this instance, $\P(X_i=n)=\beta_n/\lambda$
and $M$ is Poisson with rate $\lambda$.  

\begin{figure}[h]
\centering
\includegraphics[scale=0.35]{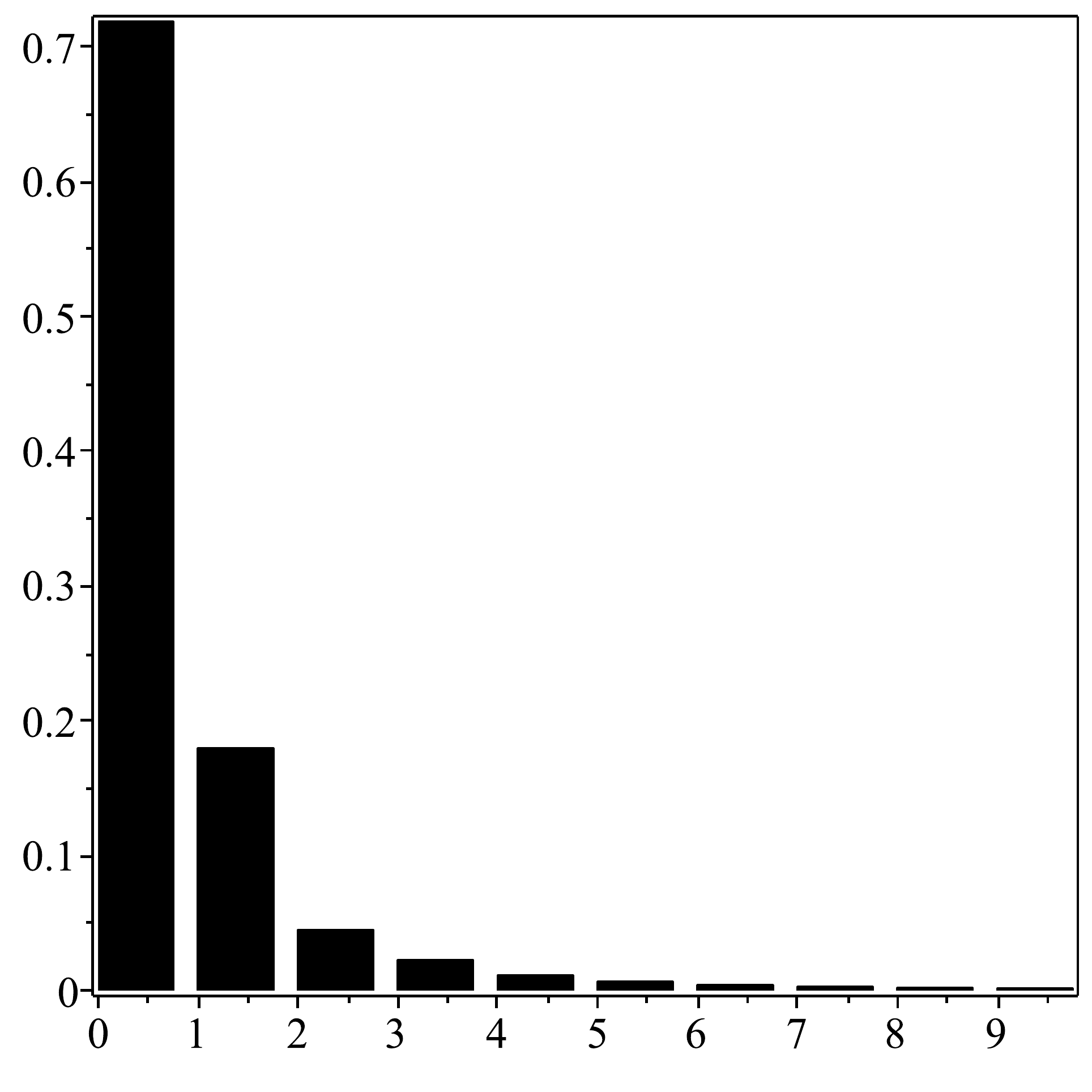}
\caption{
The infinitely divisible 
tournament distribution, $p_n=e^{-\lambda}S_n/4^n$.
}
\label{F_td}
\end{figure}

Let us also observe that, alternatively, \cref{T_main} can be proved
slightly more 
directly from \eqref{E_CDFS22}
using \cite[Theorem~3.1]{HJ78}. 
Indeed, put 
\[
a_n=\frac{S_n}{4^n},\quad\quad b_n=\frac{N_n}{4^n}.
\]
Then note that 
\[
na_n=\sum_{k=1}^n  b_k a_{n-k}
\]
and that, by \eqref{E_Ale08}, $b_n$ is regularly  
varying with index $p=-3/2$. 
See \cite{HJ78} for definitions and more details. 

\subsection{The constant}\label{S_constant}

In this section, we obtain a numerical approximation to the 
constant factor in the leading order asymptotics of $S_n$, verifying
the asymptotics conjectured in \cite{Tak86}. 

\begin{proof}[Proof of \cref{C_main}]
By \cref{T_main}, we have that 
\[
\frac{n^{5/2}S_n}{4^n}\sim
\frac{n^{3/2}N_n}{4^n}
e^\lambda.
\]
Next, we use the technical estimates in \cref{L_tech}, proved
in \cref{A_tech} below. By \eqref{E_Nn}, it follows that 
\[
\frac{n^{5/2}S_n}{4^n}\to
\frac{1}{2\sqrt{\pi}}
e^\lambda.
\]
Using values available at \cite{A145855}, we find that 
\[
\sum_{k=1}^{100}\frac{N_k}{k4^k}\doteq0.3300510246. 
\]
Hence, by \eqref{E_tailNn}, it follows that 
\begin{equation}\label{E_num}
0.330235
\le
\lambda 
\le 
0.330239, 
\end{equation}
and so, 
\[
\frac{n^{5/2}S_n}{4^n}\to 0.392\ldots, 
\]
as conjectured in \cite{Tak86}. 
\end{proof} 

\subsection{The exponential term}\label{S_exp}

Next, we obtain a probabilistic interpretation
of the exponential term in \eqref{E_asySn}, 
which is more 
closely related (on the face of it) to score sequences. 
Recall (see \cref{S_irr}) that $S_{n,m}$ is the number of score sequences
with exactly $m$ irreducible subscores. 

\begin{prop}\label{P_ID}
For all $n\ge1$, we have 
\[
N_n=n\sum_{m=1}^n\frac{1}{m}S_{n,m}.
\]
\end{prop}

For example, when $n=6$, we have 
$S_{6,1}=7$, $S_{6,2}=7$, $S_{6,3}=3$, $S_{6,4}=4$, $S_{6,5}=0$ and $S_{6,6}=1$.
Hence 
\[
6\sum_{m=1}^{6} \frac{S_{6,m}}{m}=
42+21+6+6+0+1=76.
\]
On the other hand, 
we obtain the same total sum by computing 
\[
N_6=\sum_{i=1}^{6}\frac{(-1)^{6+d}}{12}{2d\choose d}
=-\frac{1}{6}+\frac{1}{2}-\frac{5}{3}+\frac{1}{2}-\frac{1}{6}+77
=76. 
\]

Following \cite{HJ78}, 
we call a positive sequence $\{1=a_0,a_1,\ldots\}$ 
a {\it renewal sequence} if, for some other sequence $\{0=b_0,b_1,\ldots\}$, 
we have that 
\begin{equation}\label{E_RS}
A(x)=\frac{1}{1-B(x)}
\end{equation}
where $A(x)=\sum_{n=0}^\infty a_nx^n$ and 
$B(x)=\sum_{n=1}^\infty b_nx^n$ are their generating functions. 
As noted in \cite{HJ78} (p.\ 66), any such sequence is infinitely divisible. 
More specifically, let us note the following. 

\begin{lemma}[\hspace{1sp}\cite{HJ78}]
\label{L_RS}
Suppose that $\{1=a_0,a_1,\ldots\}$ is a renewal sequence 
satisfying \eqref{E_RS}. Then  $\{1=a_0,a_1,\ldots\}$
is infinitely divisible and 
the log transform of $a_n$ is 
\[
\hat a_n=n\sum_{m=1}^n\frac{1}{m}[x^n][B(x)^m]. 
\]
\end{lemma}

\begin{proof}
Taking logs on both sides of \eqref{E_RS} and then differentiating,
we find that \[
A'(x)=A(x)\sum_{m=1}^\infty \frac{1}{m}\frac{d}{dx}B(x)^{m}.
\]
Then, comparing coefficients, 
\[
na_n=\sum_{k=1}^n \left(k\sum_{m=1}^k\frac{1}{m}[x^k][B(x)^m]\right)a_{n-k}. \qedhere
\]
\end{proof}

\begin{proof}[Proof of \cref{P_ID}]
Setting $S_0=1$ and $S_{0,m}=0$, we have 
\begin{enumerate}
\item $S_n=\sum_{m=1}^n S_{n,m}$, for $n\ge1$, and
\item $S_{n,m}=\sum_{k=1}^n S_{k,1}S_{n-k,m-1}$,  
for $1<m\le n$. 
\end{enumerate}
Therefore, if we let $S(x)=\sum_n S_{n}x^n$ and $S_m(x)=\sum_n S_{n,m}x^n$ 
denote their generating functions, 
then by induction 
\[
S_m(x)=S_1(x)S_{m-1}(x)=[S_1(x)]^m,
\] 
and so 
\begin{equation}\label{E_RS}
S(x)=1+\sum_{m=1}^\infty S_m(x)=\frac{1}{1-S_1(x)}.
\end{equation}
Hence, the result follows by \eqref{E_CDFS22}
and  \cref{L_RS}, noting that 
$S_{n,m}=[x^n][S_1(x)]^m$. 
\end{proof}

\begin{proof}[Proof of \cref{T_irr}]
By \cref{T_main}
and \cref{P_ID}, 
\[
\E\left(\frac{1}{{\mathcal I}_n}\right)
=\sum_{m=1}^n\frac{1}{m}\frac{S_{n,m}}{S_n}
=\frac{N_n}{nS_n}\to
e^{-\lambda}.
\]
Using the numerics \eqref{E_num}, in the proof of \cref{C_main}
above, it follows that 
\[
\E\left(\frac{1}{{\mathcal I}_n}\right)\to 0.718\ldots, 
\]
as claimed. 
\end{proof}

\subsection{Strong scores}
In this section, we prove \cref{T_irr2}.
See, e.g., Feller \cite[Section XIII]{Fel68} for background
on discrete renewal theory. Following the notation there, 
we put  
\[
u_n=\frac{S_n}{4^n},\quad\quad f_n=\frac{S_{n,1}}{4^n}.
\] 
Then by \eqref{E_RS}
\begin{equation}\label{E_RS2}
U(s)=\frac{1}{1-F(s)}, 
\end{equation}
where 
$U(s)=\sum_{n=0}^\infty u_n s^n$
and $F(s)=\sum_{n=1}^\infty f_n s^n$
are the associated generating functions.  
Note that, by \eqref{E_CDFS22} and the numerics in 
\cref{T_irr}, we have that 
\begin{equation}\label{E_USN}
U(1)=S(1/4)=
e^\lambda 
>1.
\end{equation}
Hence $F(1)<1$. 
We extend the ``defective'' distribution 
$\{f_n:n\ge1\}$ to a probability distribution by setting 
 $f_\infty=1-F(1)>0$.

This relationship 
(in the notation of \cite{AB16}) 
corresponds to a ``transient'' renewal 
process 
\[
\tau =\{0=\tau_0,\tau_1,\tau_2,\ldots\}
\]
with independent and identically distributed
inter-arrival times $\Delta\tau_i$ 
distributed as $\P(\tau_1=n)=f_n$
and $\P(\tau_1=\infty)=f_\infty$.
By \eqref{E_RS2} it follows that $u_n=\P(n\in \tau)$.
Such a renewal process almost surely 
terminates (that is, eventually some $\Delta\tau_i=\infty$) after some finite amount of time, and  in 
that sense is transient. 

There is interest in the asymptotics of sequences 
$u_n$ and $f_n$
related by \eqref{E_RS}. See, e.g., the series of works by 
de Bruijn and Erd\H{o}s \cite{deBE51,deBE52,deBE53}. 
Most results in the literature give information about the asymptotics of $u_n$, 
assuming certain conditions on the regularity of $f_n$ (which, as discussed in \cite{AB16}, 
can in practice be difficult to verify). However, in our current situation, 
having already established the asymptotics of $S_n$, 
we want to go in the other direction. 
To this end, the 
following 
``reverse renewal theorem'' 
is just what we need. 

\begin{thm}[\hspace{1sp}{\cite[Theorem 1.4]{AB16}}]
\label{T_AB16}
If $u_n$ is regularly varying and $\tau$ is transient, then 
$f_n/u_n\to f_\infty^2$ as $n\to\infty$. 
\end{thm}

As discussed in \cite{AB16}, 
this result can also be proved using techniques from 
Banach algebra \cite{CNW73,DK11}.
See \cite{AB16} for a short probabilistic proof.

\begin{proof}[Proof of \cref{T_irr2}]
By \cref{T_main,C_main}, the sequence $u_n$ is regularly varying with 
index $-5/2$. In fact, $n^{5/2}u_n$ converges. 
Therefore, applying \cref{T_AB16} and \eqref{E_USN}, it follows that 
\[
\frac{S_{n,1}}{S_n}=\frac{f_n}{u_n}
\to 
f_\infty^2
=\frac{1}{S(1/4)^2}
=e^{-2\lambda}, 
\]
and so by \cref{C_main} 
\[
\frac{n^{5/2}S_{n,1}}{4^n}
\to 
\frac{e^{-\lambda}}{2\sqrt{\pi}},
\]
as required. 
\end{proof}

\subsection{Limiting distribution}
Finally, we prove \cref{T_Snm}, which extends 
the result  of \cref{T_irr2}. To this end, we invoke the following special case
of a (much more general) result in \cite{CNW73}
(cf.\ \cite[Chapter IV]{AN72}). 

\begin{thm}[\hspace{1sp}{\cite[Theorem 1]{CNW73}}]
\label{T_CNW73}
Let $\{\mu_1,\mu_2,\ldots\}$ be a probability measure
for which $\mu_{n+1}/\mu_n\to 1$, as $n\to\infty$, and 
so that, for all sufficiently large $n$, 
\begin{equation}\label{E_supC}
\max_{k\le n/2} \frac{\mu_{n-k}}{\mu_n}\le C,  
\end{equation}
for some constant $C$. 
Then, for any $m\ge2$, 
we have that 
$\mu_n^{*m}/\mu_n\to m$, 
as $n\to\infty$. 
\end{thm}

\begin{proof}[Proof of \cref{T_Snm}]
Note that, by \eqref{E_RS2} and \eqref{E_USN}, 
\[
\mu_n=\frac{S_{n,1}}{4^n}\frac{1}{1-e^{-\lambda}},
\quad n\ge1, 
\]
is a probability measure. 
By \cref{T_irr2}, we have that 
$\mu_{n+1}/\mu_n\to 1$. 
Moreover, 
\[
\mu_n\sim 
\frac{1}{2\sqrt{\pi}}
\frac{e^{-\lambda}}{1-e^{-\lambda}}
\frac{1}{n^{5/2}},
\]
and so condition \eqref{E_supC} is clear. 
Hence, by \cref{T_CNW73}, 
\[
(1-e^{-\lambda})^{-(m-1)}
\frac{S_{n,m}}{S_{n,1}}
=\frac{\mu_n^{*m}}{\mu_n}
\to m.
\]
Therefore, by \cref{T_irr2}, 
\[
\frac{S_{n,m}}{S_{n}}\to m(1-e^{-\lambda})^{m-1}e^{-2\lambda},
\]
as claimed. 
\end{proof}

\appendix

\section{Technical estimates}\label{A_tech}

The following approximations are used in the
numerical estimates above.

\begin{lemma}\label{L_tech}
For all $n\ge10$, we have 
\begin{equation}\label{E_Nn}
\frac{1}{2\sqrt{\pi}}\frac{4^n}{n^{3/2}}\left(1-\frac{1}{4n}\right)
\le N_n \le \frac{1}{2\sqrt{\pi}}\frac{4^n}{n^{3/2}}
\end{equation}
and 
\begin{equation}\label{E_tailNn}
\frac{1}{3\sqrt{\pi}}\frac{1}{n^{3/2}}
\left(1-\frac{2}{n}\right)
\le
\sum_{k>n} \frac{N_k}{k4^k}\le \frac{1}{3\sqrt{\pi}}\frac{1}{n^{3/2}}.
\end{equation}
\end{lemma}

\begin{proof}
The central binomial coefficients satisfy
\[
\frac{4^n}{\sqrt{\pi n}}\left(1-\frac{1}{8n}\right)\le {2n\choose n}\le 
\frac{4^n}{\sqrt{\pi n}}\left(1-\frac{1}{9n}\right).
\]
Therefore, by \eqref{E_Ale08}, we have 
\begin{align*}
N_n&\le \frac{1}{2n}\sum_{k=1}^n{2d\choose d}\\
&\le\frac{1}{2n}\left({2n\choose n}+\frac{n}{2} {n\choose \lfloor n/2\rfloor}\right)\\
&\le \frac{1}{2\sqrt{\pi}}\frac{4^n}{n^{3/2}}\left(1-\frac{1}{9n}+\frac{n}{2^n}\right)\\
&\le \frac{1}{2\sqrt{\pi}}\frac{4^n}{n^{3/2}}
\end{align*}
for all $n\ge10$. Similarly, 
\[
N_n\ge
\frac{1}{2\sqrt{\pi}}\frac{4^n}{n^{3/2}}\left(1-\frac{1}{8n}-\frac{n}{2^n}\right)
\ge \frac{1}{2\sqrt{\pi}}\frac{4^n}{n^{3/2}}\left(1-\frac{1}{4n}\right)
\]
for all $n\ge10$, and \eqref{E_Nn} follows. 

Next, observe that, for $n\ge10$, we have by \eqref{E_Nn}
that 
\[
\sum_{k>n}\frac{N_k}{k4^k}\le \frac{1}{2\sqrt{\pi}}\int_n^\infty \frac{dx}{x^{5/2}}
=\frac{1}{3\sqrt{\pi}}\frac{1}{n^{3/2}}
\]
and 
\begin{align*}
\sum_{k>n}\frac{N_k}{k4^k}&\ge \frac{1}{2\sqrt{\pi}}\left(\int_{n+1}^\infty \frac{dx}{x^{5/2}}
-\frac{1}{4}\int_{n}^\infty \frac{dx}{x^{7/2}}\right)\\
&=\frac{1}{3\sqrt{\pi}}\frac{1}{n^{3/2}}\left(
\left(\frac{n}{n+1}\right)^{3/2}-\frac{3}{20n}\right)\\
&\ge \frac{1}{3\sqrt{\pi}}\frac{1}{n^{3/2}}
\left(1-\frac{2}{n}\right), 
\end{align*}
and so \eqref{E_tailNn} follows. 
\end{proof}

\providecommand{\bysame}{\leavevmode\hbox to3em{\hrulefill}\thinspace}
\providecommand{\MR}{\relax\ifhmode\unskip\space\fi MR }
\providecommand{\MRhref}[2]{%
  \href{http://www.ams.org/mathscinet-getitem?mr=#1}{#2}
}
\providecommand{\href}[2]{#2}

\end{document}